\documentclass{amsart}

\usepackage{graphicx}%
\usepackage{multirow}%
\usepackage{amsmath,amssymb,amsfonts}%
\usepackage{amsthm}%
\usepackage{mathrsfs}%
\usepackage[title]{appendix}%
\usepackage{xcolor}%
\usepackage{textcomp}%
\usepackage{manyfoot}%
\usepackage{booktabs}%
\usepackage{algorithm}%
\usepackage{algorithmicx}%
\usepackage{algpseudocode}%
\usepackage{listings}%
\usepackage{longtable}
\usepackage{nicematrix,tikz}

\newtheorem{theorem}{Theorem}[section]

\theoremstyle{definition}

\theoremstyle{remark}

\numberwithin{equation}{section}

\newcommand{\diag}{\operatorname{diag}}

\begin{document}

\title[Zonal polynomials]{An explicit formula for zonal polynomials}


\author[Haoming Wang]{Haoming Wang}
\address{School of Mathematics, Sun Yat-sen University, Guangzhou {\rm510275}, China}
\email{wanghm37@mail2.sysu.edu.cn}
\thanks{}


\date{}

\dedicatory{}

\begin{abstract} 
The derivation of zonal polynomials involves evaluating the integral 
\[
\exp\left( - \frac{1}{2} \operatorname{tr} D_{\beta} Q D_{l} Q \right)
\]
with respect to orthogonal matrices \(Q\), where \(D_{\beta}\) and \(D_{l}\) are diagonal matrices. The integral is expressed through a polynomial expansion in terms of the traces of these matrices, leading to the identification of zonal polynomials as symmetric, homogeneous functions of the variables \(l_1, l_2, \ldots, l_n\). The coefficients of these polynomials are derived systematically from the structure of the integrals, revealing relationships between them and illustrating the significance of symmetry in their formulation. Furthermore, properties such as the uniqueness up to normalization are established, reinforcing the foundational role of zonal polynomials in statistical and mathematical applications involving orthogonal matrices.
\end{abstract}

\maketitle
\section{Introduction}

The zonal polynomial was introduced by Hua and James during the 1950s and 1960s in the study of the integral properties of orthogonal invariant polynomials on classical domains such as symmetric matrices. However, an explicit formula for zonal polynomials had not been derived until 1965 by Tumura, who expressed
\begin{equation*}
    \left[\operatorname{tr}AHBH^{\prime}\right]^{f}
\end{equation*}
as a homogeneous, symmetric polynomial in the latent roots of $A$ and $B$ separately, now known as zonal polynomials. This formula then had been neglected since 1968 by the replacement of a recursion formula found by James\cite{james1968calculation} using Laplace-Beltrami operator.  

The integrals over orthogonal groups studied by Hua and James reflect the importance of zonal polynomials in harmonic analysis and group theory. Key relationships between these integrals and zonal polynomials include
\begin{eqnarray}
    & \int_{O(n)} \left[\operatorname{tr} AH\right]^{f} (dH) = \sum_{\kappa} \frac{\chi_{2\kappa} (1)}{Z_{\kappa}(I_{n})} Z_{\kappa} (AA^{\prime})\\
    & \left[\operatorname{tr} A\right]^{f} = \frac{2^{f}f!}{(2f)!} \sum_{\kappa} \chi_{2\kappa}(1) Z_{\kappa}(A)\\
    & \int_{O(n)} Z_{\kappa}(AHBH^{\prime}) (dH) = \frac{Z_{\kappa} (A)Z_{\kappa} (B)}{Z_{\kappa}(I_{n})}
\end{eqnarray}
Hua\cite{hua1963harmonic} calculated the integral of (1.1) and didn't find the coefficients of $Z_{\kappa}(A)$ explicitly; James\cite{james1961zonal} derived the identities (1.2) and (1.3) but also failed in the computation. Tumura\cite{Tumura1965THEDO} evaluated the integral of (1.3) by decomposing orthogonal matrices into products of elementary rotation matrices. However, complications arise because an orthogonal matrix with determinant \(-1\) cannot always be expressed solely as a product of rotations, making calculations for complex integrals involving orthogonal matrices more subtle. This difficulty reflects broader challenges in classifying the representations of general linear groups, where zonal polynomials appear as linear combinations of group characters.

The theory of symmetric group representations, foundational to the development of zonal polynomials, was pioneered by the works of Littlewood and Richardson\cite{littlewood1934group}, Murnaghan\cite{murnaghan1937representations} and Thrall\cite{thrall1942symmetrized}. Foulkes\cite{foulkes1950concomitants} proposed a conjecture about the composition of two representations, called by Littlewood the plethysm, which is recently proved by Bowman and his coauthors\cite{bowman2024110}\cite{bowman2023partitionalgebraplethysmcoefficients}. The importance of zonal polynomials was further developed by Hua\cite{hua1955severalen}, James and Constantine\cite{james1961zonal, james1968calculation, constantine1966hotelling}. Hua's 1958 book\cite{hua1963harmonic} introduced a linear recursion between group characters and zonal polynomials. Meanwhile, James\cite{james1968calculation} calculated the zonal polynomials up to $f= 6$ degrees. Later on, Davis\cite{davis1980invariant} extended these results to two matrix arguments and Michel\cite{Michel2008} evaluated zonal polynomials for identity matrices. Hayashi\cite{hayashi2015decomposition} proposed a decomposition rule, and Colmenarejo and Rosas\cite{colmenarejo2015combinatorics} discussed a reduced Kronecker coeffiecient similar to our method that will be explained in the following. Modern treatments by Faraut and Kor\'anyi\cite{faraut1994analysis} and Macdonald\cite{macdonald1998symmetric} provide comprehensive understanding in symmetric functions, including zonal polynomials.

The development of zonal polynomials spans decades, strongly connected to harmonic analysis and special functions. In 1955, Herz\cite{herz1955} extended classical Bessel functions to matrix arguments, laying a foundation for future works. Gelbart\cite{gelbart1974theory} introduced Stiefel harmonics, followed by Ton That\cite{ton1976lie} who studied Lie group representations and harmonic polynomials for matrix variables. Gross and Richard\cite{gross1987special} studied special functions like hypergeometric functions. Adams\cite{Adams1998LiftingOC} advanced characteristic theory with orthogonal and metaplectic groups. Howe and Zhu\cite{Howe2002EigendistributionsFO} explored eigendistributions for the orthogonal groups and symplectic group, while Grinberg and Rubin\cite{grinberg2004radon} developed Radon inversion techniques for Gamma functions. Balderrama et al.\cite{balderrama2005formula} and Aristidou et al.\cite{aristidou2006differential} examined recursion relations for Laguerre polynomials. Miyazaki\cite{Miyazaki2011OnBI} extended Bessel integrals in degenerate cases. 

The significance of zonal polynomials extends well beyond pure mathematics into multivariate statistics, group representations, and physics. Let’s begin with some examples.

{\it Example 1:} (Wishart Distribution). The Wishart distribution describes the distribution of sample covariance matrices. If \(X\) is an \(n \times p\) matrix with independent and identically distributed normal rows \(N(\mu, \Sigma)\) for some mean vector $\mu$ covariance matrix \(\Sigma\), the sample covariance matrix \(S\) is given by
\[S = \frac{1}{n-1} X^{\prime} X,\]
The distribution of \(S\) is a Wishart distribution, denoted as
\[S \sim W_p(n, \Sigma),\]
where \(p\) is the number of variables and \(n\) is the sample size. The distribution of its eigenvalues is expressed in terms of zonal polynomials.

{\it Example 2:} (Group Characters). In group representation theory, the system of all monomials $z_{1}^{l_{1}}z_{2}^{l_{2}} \dots z_{n}^{l_{n}}$ is a complete system in the space of functions analytic on the space of symmetric matrices and the monomials of different degrees (i.e., $l_{1} + \dots + l_{n} = l_{1}^{\prime} + \dots + l_{n}^{\prime})$ are orthogonal to each another. Let $z= (z_{1},z_{2}, \dots, z_{n})$, and let $z^{[l]}$ be the vector with components
\[\sqrt{\frac{l!}{l_{1}! l_{2}!\dots l_{n}! } } z_{1}^{l_{1}}z_{2}^{l_{2}} \dots z_{n}^{l_{n}}\]
in $n(n + 1)\dots (n + l - 1) / l!$-dimensional space. The transformation $w = zU$ induces a transformation
\[w^{[l]} = z^{[l]} U^{[l]},\]
where $U^{[l]}$ is the $l$th symmetric Kronecker power of $U$.  
\[\sigma (U^{[2]})^{[l]} = \sum \chi_{2l_{1},2l_{2},\dots,2l_{n}}(U),\]
where $\chi_{2l_{1},2l_{2},\dots,2l_{n}}$ is the character (i.e., the trace of representation matrix) of the irreducible representation of $GL(n) \curvearrowright$ $\{\text{the set of all $n \times n$ symmetric matrices}\}$ according to $(2l_{1},2l_{2},\dots,2l_{n})$ and $\sigma$ represents for trace.\cite[p. 181]{weyl1946classical} These characters $\chi_{2l_{1},2l_{2},\dots,2l_{n}}$ are determined by a linear equation system of zonal polynomials.

{\it Example 3:} (Physics). In physics, integrals over the orthogonal group are essential for understanding systems with rotational symmetry. For example,
\begin{itemize}
    \item[1.] (Quantum Mechanics). In quantum mechanics, Rotational symmetry in wavefunctions requires evaluating integrals over the orthogonal group $O(m)$
    \[\langle \psi | H | \phi \rangle = \int \psi^*(U) H \phi(U) (dU),\]
    where \(H\) is a unitary self-adjoint operator. This ensures consistent physical predictions across coordinate systems.

    \item[2.] (Statistical Mechanics). In statistical mechanics, The partition function for rotationally symmetric systems is
    \[
    Z = \int e^{-\beta H(U)} dU,
    \]
    where \(H\) is the Hamiltonian and \(dU\) is the measure over phase space.  These integrals help derive thermodynamic properties like free energy and entropy.

    \item[3.] (Field Theory). In chiral Lagrangians, integrals describe pion dynamics
    \[\int \frac{F^2}{4} \operatorname{tr}(\partial_\mu U^\dagger \partial^\mu U) dU,\]
    where \(F\) denotes the decay constant. These integrals help compute correlation functions and scattering amplitudes.
\end{itemize}

Therefore understanding integrals over the orthogonal group contributes significantly to both theoretical and computational mathematics. Our approach combines methods from Paget and Wildon\cite{Rowena2016minimal} and Orellana and Zabrocki\cite{Orellana2021107943}, by representing integrals as symmetric functions of latent roots and computing coefficients based on minimal and maximal weights. Zonal polynomials facilitate the comparison of specific coefficients related to the eigenvalues of matrices $A$ and $B$. By decomposing the orthogonal matrix $H$ into elementary rotations and reflections, the integral is evaluated term by term, simplifying computations and revealing underlying algebraic structures within the orthogonal group.

\section{Characters of symmetric groups}
Let $n \geq 2 $ be an integer.
\[D(x_1, \ldots, x_n) = \prod_{1\leq i< j \leq n} (x_{i} - x_{j}).\]
\begin{theorem}
    Let \( h_1(x), \ldots, h_n(x) \) be real analytic functions. Then, 
\[
\lim_{\substack{x_1 \rightarrow x \\ x_n \rightarrow x}} \frac{\left| 
\begin{array}{ccc}
h_1(x_1) & \ldots & h_n(x_1) \\ 
\vdots & \ddots & \vdots \\ 
h_1(x_n) & \ldots & h_n(x_n) 
\end{array} 
\right|}{D(x_1, \ldots, x_n)} = \frac{(-1)^{\frac{1}{2}n(n-1)}}{1! \, 2! \, \ldots \, (n-1)!} 
\left| 
\begin{array}{ccc}
h_1(x) & \ldots & h_n(x) \\ 
h_1'(x) & \ldots & h_n'(x) \\ 
\vdots & \ddots & \vdots \\ 
h_1^{(n-1)}(x) & \ldots & h_n^{(n-1)}(x) 
\end{array} 
\right|.
\]
\end{theorem}

\begin{proof}
        Without loss of generality, assume the limits \( x_1, \ldots, x_n \) approach 0. The theorem follows from the Maclaurin series with remainder terms.
\end{proof}

Let $f_{1} \geq f_{2} \geq \dots\geq f_{n}\geq 0$ be integers. We introduce the notation 
\[
M_{f_1, \ldots, f_n}(x_1, \ldots, x_n) = \left| 
\begin{array}{ccc}
x_1^{f_1 + n - 1} & \ldots & x_n^{f_1 + n - 1} \\ 
x_1^{f_2 + n - 2} & \ldots & x_n^{f_2 + n - 2} \\ 
\vdots & \ddots & \vdots \\ 
x_1^{f_n} & \ldots & x_n^{f_n} 
\end{array} 
\right|,
\]
It is clear that
\[M_{0, \ldots, 0}(x_1, \ldots, x_n) = D(x_1, \ldots, x_n).\]

We set
\[N(f_1, \ldots, f_n) = \lim_{\substack{x_1 \rightarrow x \\ x_n \rightarrow x}} \frac{M_{f_1, \ldots, f_n}(x_1, \ldots, x_n)}{D(x_1, \ldots, x_n)}.\]
From Theorem 2.1, we have
\[
N(f_1, \ldots, f_n) = \frac{D(f_1 + n - 1, f_2 + n - 2, \ldots, f_n)}{D(n - 1, n - 2, \ldots, 1, 0)}.
\]

\section{Polar coordinates for orthogonal matrices}

Let 
\[
V(\theta) = \begin{pmatrix}
	 \cos(\theta) & \sin(\theta) \\
	-\sin(\theta) & \cos(\theta) 
 \end{pmatrix},
\]
\begin{theorem}
Every orthogonal matrix \( \Gamma \) has the decomposition
\begin{equation}
    \Gamma = \prod_{i=1}^{n} U_{i}^{0,1}\prod_{j=i}^{n-1}V_{i}(\theta_{ij})
    \label{eq: orthogonal decomposition}
\end{equation}
where $V_{i}(\theta) = \diag (I_{i-1}, V(\theta),  I_{n-i-1} )$ is an orthogonal matrix that rotates by an angle \( \theta \) in the \( (i, i+1) \)-plane and $U_{i}= \diag (I_{i-1}, -1,  I_{n-i} )$ is an orthogonal matrix that reflects along the \( i \)-axis.    
\end{theorem}
\begin{figure}[htbp]
    \centering
    \includegraphics[width=\linewidth]{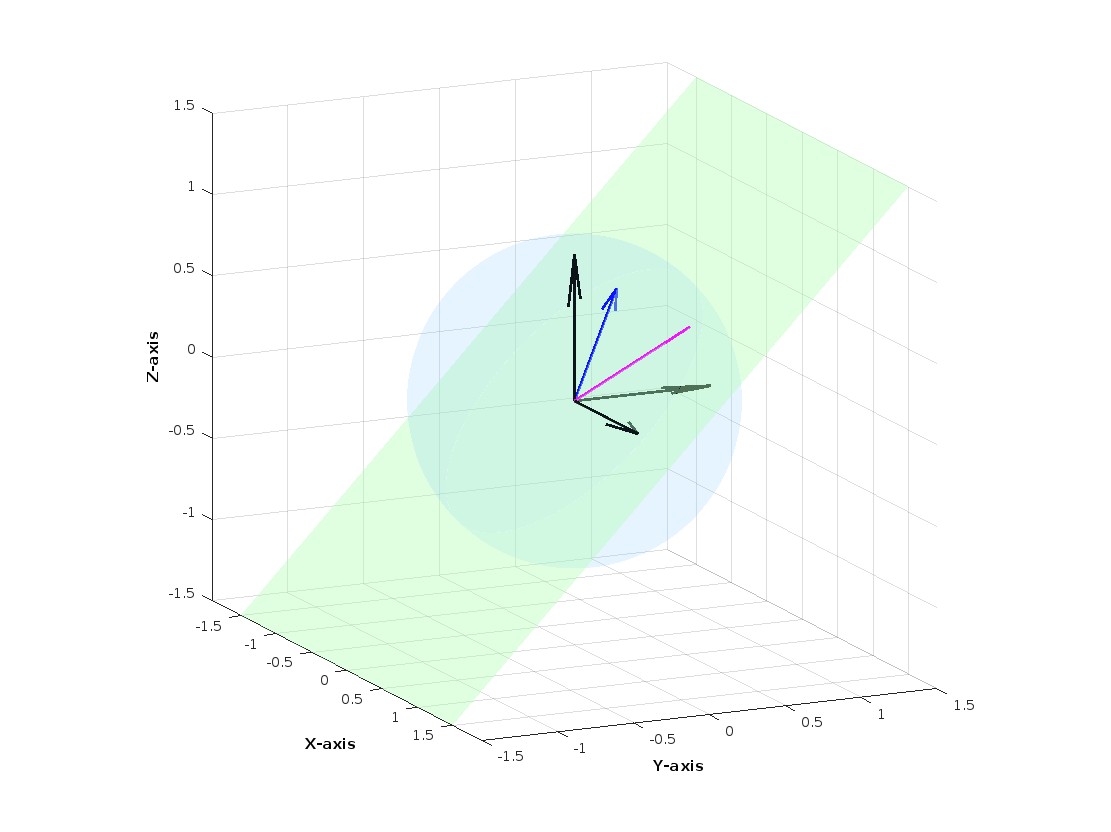}
    \caption{Orthogonal decomposition in terms of rotation and reflection. The blue one is the original vector and the red one is the target vector under rotation and reflection.}
    \label{fig:enter-label}
\end{figure}
\begin{proof}
     Every orthogonal matrix in \( n \)-dimensional space can be decomposed as a product of rotation and reflection. The rotation matrix can be performed by rotations in the \( (i, i+1) \) planes, while the reflection matrix can be decomposed into rotations and a one-dimensional reflection. Then induction assures this theorem.
\end{proof}

Now we consider polar coordinates for real-symmetric matrices. Any real symmetric matrix \( T \) can be decomposed in the form
\begin{equation}
    T = \Gamma A \Gamma^{\prime},
    \label{eq: orthogonal matrix}
\end{equation}
where \( \Gamma \) is a real orthogonal matrix with determinant \( +1 \) and \( A = \operatorname{diag}(\lambda_1, \ldots, \lambda_n) \) with \( \lambda_1 > \lambda_2 > \ldots > \lambda_n \). It is easy to see that there is in general a one-to-one correspondence between real symmetric matrices and $\{O\} \times \Gamma$, where $\{O\}$ is the set of left cosets of the group of orthogonal matrices with determinant $+1$ with respect to the subgroup of all diagonal matrices of the form $\diag(\pm 1)$.

Differentiating (\ref{eq: orthogonal matrix}) we obtain 
\[
\Gamma^{\prime} dT \Gamma = \delta \Gamma \cdot \Lambda + d\Lambda - \Lambda \delta \Gamma,
\]
where \( \delta \Gamma = \Gamma^{-1} d\Gamma \) is a skew-symmetric such that
\[\sigma(dT \cdot dT^{\prime}) = \sigma\left\{ (\delta \Gamma \cdot \Lambda - \Lambda \delta \Gamma )^2 \right\} + \sigma(d\Lambda \cdot d\Lambda^{\prime}),\]
where $\sigma(A)$ represents for the trace of the matrix $A$.

Thus, the volume element in terms of eigenvalues is
\[
\dot{T} = 2^{\frac{n(n-1)}{2}} \prod_{i < j} |\lambda_i - \lambda_j| \, d\lambda_1 \cdots d\lambda_n \dot{[O]},
\] 
where $\dot{[O]} = \prod_{1\leq i < j \leq n}\delta\gamma_{ij}$.

\begin{theorem} The volume element has the form
   \[ \dot{[O]} = \prod_{i=1}^{n-2} \prod_{j=i}^{n-2} \sin(\theta_{ij})^{n-j-1}\prod_{i=1}^{n-1} \prod_{j=i}^{n-1} d\theta_{ij},\] 
   where is defined by (\ref{eq: orthogonal decomposition}).
\end{theorem}
\begin{proof}
    Induction.\cite[Theorem 2.1]{Tumura1965THEDO}
\end{proof}

\section{Derivations of zonal polynomials}
\subsection{Orthonormal systems in the space of symmetric orthogonal matrices}
Let $\Re_{\text {II}}$ and $\mathfrak{C}_{\text {II}}$ be circular domain of symmetric matrices and characteristic manifold considered in \cite{hua1963harmonic}. We shall consider orthonormal systems in the spaces $\Re_{\text {II}}$ and $\mathfrak{C}_{\text {II}}$. The transformation
\begin{equation}
T=U S U^{\prime}
\label{eq: congruence}
\end{equation}
with unitary matrix $U$ maps the symmetric unitary matrix $S$ into the symmetric unitary matrix $T$. We shall arrange the elements of the matrix $S$ as a vector $s$
$$\begin{aligned}
&\left(s_{11}, \sqrt{2} s_{12}, \sqrt{2} s_{13}, \ldots,\right. \sqrt{2} s_{1 n}, s_{22}, \sqrt{2} s_{23}, \ldots \\
&\left.\ldots, \sqrt{2} s_{2 n}, \ldots, s_{n-1, n-1}, \sqrt{2} s_{n-1, n}, s_{n n}\right)
\end{aligned}$$
of dimension $n(n+1) / 2$, and we shall carry out a similar operation for the matrix $T$, denoting the vector so obtained by $t$. When the matrix $S$ is mapped into the matrix $T$ by the transformation (4.1), the vector $s$ is mapped into the vector $t$ by a linear transformation with matrix $U^{[2]}$ of order $n(n+1) / 2$.

From the vector $s$ we shall construct the vector $s^{[f]}$ of dimension
$$
\frac{(n(n+1) / 2+f-1)!}{f!(n(n+1) / 2-1)!}
$$

The matrix of the transformation of $s^{[f]}$ into $t^{[f]}$ is
\begin{equation}
\left(U^{[2]}\right)^{[f]}
\label{eq: plethysm}
\end{equation}

The components of $s^{[f]}$ are monomials of $s_{i j}$ of degree $f$. These components are linearly independent, and any homogeneous polynomial of degree $f$ can be written in the form of a linear combination of the components.

It is known that the space of homogeneous polynomials of degree $f$ of $s_{i j}$ can be decomposed into a direct sum of subspaces which are invariant under the transformation (\ref{eq: plethysm}). These subspaces have dimension $N\left(2 f_{1}, \cdots, 2 f_{n}\right)$, where $f_{1}+\cdots+f_{n}=f$. The polynomials which form the basis of an $N\left(2 f_{1}, \cdots, 2 f_{n}\right)$-dimensional subspace we shall denote by
\begin{equation}
    \varphi_{f_{1}, \ldots, f_{n}}^{(i)}(S), \quad i=1,2, \ldots, N\left(2 f_{1}, \ldots, 2 f_{n}\right)
    \label{eq: zonal1}
\end{equation}

When $S$ is mapped into $T$ by the transformation (\ref{eq: congruence}), $\varphi_{f_{1}, \ldots, f_{n}}^{(i)}(S)$ is mapped into $\varphi_{f_{1}, \cdots, f_{n}}^{(i)}(T)$ by the transformation with matrix $A_{2 f_{1}, \cdots, 2 f_{n}}(U)$. Moreover, system (\ref{eq: zonal1}) is orthogonal on $\mathfrak{C}_{\text {II}}$, i.e.

$$\int_{\mathfrak{C}_{\text{II}}} \varphi_{f_{1}, \ldots, f_{n}}^{(i)}(S) \overline{\varphi_{g_{1}, \ldots, g_{n}}^{(j)}(S)} = \delta_{ij} \delta_{fg} \rho_{f}$$
where $\dot{S}$ is the volume element of the space of symmetric unitary matrices.

By setting
$$
\psi_{f_{1}, \ldots, f_{n}}^{(i)}(S)=\rho_{f_{1}, \ldots, f_{n}}^{-\frac{1}{2}} \cdot \varphi_{f_{1}, \ldots, f_{n}}^{(i)}(S)
$$
we obtain an orthonormal system in $\mathfrak{C}_{\text {II }}$. 

\subsection{Spherical harmonics} 

Let us consider the functions
\begin{equation}
\Psi_{f_{1}, \ldots, f_{n}}(S, \bar{T})=\sum_{i} \psi_{f_{1} \ldots, f_{n}}^{(i)}(S) \overline{\psi_{f_{1}, \ldots, f_{n}}^{(i)}(T)}.
\label{eq: zonal2}
\end{equation}

We directly obtain
\begin{equation*}
\int_{\mathfrak{C}_{\mathrm{II}}} \Psi_{f_{1}, \ldots, f_{n}}(S, \bar{S}) \dot{S}=N\left(2 f_{1}, \ldots, 2 f_{n}\right).
\end{equation*}
Since the $\Psi_{f_{1}, \cdots f_{n}}(S, \bar{S})$ do not depend on $S$, it follows that
\[\Psi_{f_{1}, \cdots f_{n}}(S, \bar{S}) = \frac{1}{V(\mathfrak{C}_{\text{II}})} N(2f_{1},\dots,2f_{n}).\]
Thus $\Psi_{f_{1}, \ldots f_{n}}(S, \bar{S})$ are very simple. We shall try to simplify expression (\ref{eq: zonal2}). 

Let us order lexicographically the index system $\left(f_{1}, \cdots, f_{n}\right)$. We shall write $f>g$, where $f=\left(f_{1}, \cdots, f_{n}\right), g=\left(g_{1}, \cdots, g_{n}\right)$, if $f_{1}=g_{1}, f_{2}=g_{2}, \cdots, f_{q-1}=$ $g_{q-1}, f_{q}>g_{q}$.

From (\ref{eq: congruence}) follows
\begin{equation*}
    A_{f_{1},\dots,f_{n}}(T) = A_{f_{1},\dots,f_{n}}(U)A_{f_{1},\dots,f_{n}}(S)A_{f_{1},\dots,f_{n}}(U^{\prime})
\end{equation*}
Suppose $L$ is a linear space, generated by the elements of the matrices $A_{f_{1}, \cdots, f_{n}}(S)$. When $S$ is mapped into $T, L$ is mapped into itself. This means that $L$ is an invariant subspace. This invariant subspace can be decomposed into a direct sum of invariant subspaces, corresponding to $A_{2 g_{1}, \cdots, 2 g_{n}}(X)$. Evidently, $\left(2 g_{1}, \cdots, 2 g_{n}\right) \leqq\left(2 f_{1}, \cdots, 2 f_{n}\right)$, whereby the equality certainly obtains. The expression
$$
\sigma\left[A_{f_{1}, \ldots, f_{n}}(S) A_{f_{1}, \ldots, f_{n}}(\bar{T})\right]=\sigma \left[A_{f_{1}, \ldots, f_{n}}(S \bar{T})\right]
$$
is invariant on replacing $S$ and $T$ simultaneously by $U S U^{\prime}$ and $U T U^{\prime}$. Since
$$
\sigma \left[A_{f_{1}, \ldots, f_{n}}(S \bar{T})\right]=\chi_{f_{1}, \ldots, f_{n}}(S \bar{T})
$$
it follows that
\begin{equation}
    \chi_{f_{1}, \ldots, f_{n}}(S \bar{T}) = \sum_{g \leqslant f} c_{f, g} .\Psi_{g_{1}, \ldots, g_{n}}(S, \bar{T})
    \label{eq: zonal3}
\end{equation}
whereby $c_{f, f} \neq 0$. This relation is reversible, i.e., we can write
\begin{equation}
\Psi_{f_{1}, \ldots, f_{n}}(S, \bar{T})=\sum_{g \leqslant f} d_{f, g} \chi_{g_{1}, \ldots, g_{n}}(S \bar{T}).
\label{eq: zonal4}
\end{equation}

From the last formula it is evident that
$$\Psi_{f_{1}, \ldots, f_{n}}(S, \bar{T})=\Psi_{f_{1}, \ldots, f_{n}}(S \bar{T})$$
These functions $\Psi_{f_{1}, \ldots, f_{n}}(S \bar{T})$ are called spherical harmonics in $\mathfrak{C}_{\text {II }}$.

\subsection{Calculating spherical harmonics by solving linear systems}

We shall now endeavor to find an effective method of determining the coefficients of formula (6.2.5). Since $c_{f, g}$ and $d_{f, g}$ are determined only for $g \leqq f$, we shall set $c_{f, g}=d_{f, g}=0$ for $g>f$. Moreover, since we have ordered the index systems $f=\left(f_{1}, \cdots, f_{n}\right)$, we can now consider $f$ as a simple index, i.e., $f=1,2,3, \cdots$. Then the matrices $C=\left(f_{f, g}\right)_{1}^{k}$ and $D=\left(d_{f, g}\right)_{1}^{k}$ will be nonsingular triangular matrices for any integral positive $k$.

Let us consider the integral
$$
\int_{S} \int_{T} \Psi_{f}(S \bar{T}) \chi_{g}(T \bar{S}) \dot{S} \dot{T}, \quad g \leqslant f
$$

Substituting in this integral the expression for $\chi_{g}$ given by formula (\ref{eq: zonal3}), and using (\ref{eq: zonal2}), we obtain
$$
\begin{gathered}
\sum_{h \leqslant g} c_{g, h} \int_{S} \int_{T} \Psi_{f}(S \bar{T}) \Psi_{h}(T \bar{S}) \dot{S} \dot{T} \\
=\sum_{h \leqslant g} c_{g, h} \sum_{i} \sum_{j} \int_{S} \int_{T} \psi_{f}^{(i)}(S) \overline{\psi_{f}^{(i)}(T)} \psi_{h}^{(j)}(T) \overline{\psi_{h}^{(j)}(S)} \dot{S} \dot{T}=\sum_{i, j } \sum_{h \leq g} c_{g} \delta_{i j} \delta_{f h},
\end{gathered}
$$
whence follows
\begin{equation}
\int_{S} \int_{T} \Psi_{f}(S \bar{T}) \chi_{g}(T \bar{S}) \dot{S} \dot{T}=0 \quad \text { for } \quad g<f 
\label{eq: zonal5}
\end{equation}
and
\begin{equation}
\int_{S} \int_{T} \Psi_{f}(\overline{S T}) \chi_{f}(\bar{T} \bar{S}) \dot{S} \dot{T}=N_{2f} c_{f, f} \neq 0.
\label{eq: zonal6}
\end{equation}
We shall set
\begin{equation}
\beta_{f, g}=\int_{S} \int_{T} \chi_{f}(S \bar{T}) \chi_{g}(T \bar{S}) \dot{S} \dot{T}
\label{eq: zonal7}
\end{equation}
From (\ref{eq: zonal4}), (\ref{eq: zonal5}), (\ref{eq: zonal6}) and (\ref{eq: zonal7}) we obtain
\begin{equation}
\sum_{h \leqslant g} d_{\rho, h} \beta_{h, g}=0, \quad g<f 
\label{eq: zonal8}
\end{equation}
and
\begin{equation}
\sum_{h \leqslant p} d_{f, h} \beta_{h, f}=N_{2f} c_{f, f} \neq 0
\label{eq: zonal9}
\end{equation}

It is evident that the matrix $B=\left(\beta_{h, f}\right)_{1}^{k}$ is positive definite for any integral $k \geqq 1$. Equations (\ref{eq: zonal8}) and (\ref{eq: zonal9}) can be replaced by the matrix equation
\begin{equation*}
    D B=\left(\begin{array}{cccc}N_1 c_{1,1} & * & \ldots & * \\ 0 & N_2 c_{2,2} & \ldots & * \\ \cdots & \cdots & \cdots & \cdots \\ 0 & 0 & \cdots & N_k c_{k, k}\end{array}\right)=K.
\end{equation*}
Hence it is evident that the $d_{f, g}$ can be expressed by $\beta_{f, g}$ and by the elements of the matrix $K$. Since, however, the elements of the matrix $K$ are unknown, we shall obtain from (\ref{eq: zonal9}) another system of equations from which the $d_{f, g}$ are obtained recursively in terms of $\beta_{f, g}$ and some other quantities which can be readily calculated.

We shall put
\begin{equation*}
\alpha_{f}=\int_{S} \chi_{f}(S \bar{S}) \dot{S}. 
\end{equation*}
Setting $T=S$ in (\ref{eq: zonal4}) and integrating with respect to $S$, we obtain
\begin{equation}
\sum_{h \leqslant f} d_{f, h} \alpha_{h}=N_{2f}
\label{eq: zonal10}
\end{equation}

For $f=1$, we evidently have $d_{1,1}=N_{1} / c_{1,1}$. We shall assume that for $f \leqq k-1$ the quantities $d_{f, g}(g=1,2, \cdots, f)$ are expressed by $\beta_{f, g}$ and $\alpha_{f}$, and then investigate what happens for $f=k$. From (\ref{eq: zonal8}) and (\ref{eq: zonal10}) we find
\begin{equation}
    \begin{aligned}
        \begin{pmatrix}
            d_{1,1} & 0 & \ldots & 0\\
            d_{2,1} & d_{2,2} & \ldots & 0\\
            \cdot & \cdots & \cdots & \cdot\\
            d_{k,1} & d_{k,2} & \ldots & d_{k,k}
        \end{pmatrix}
        & \begin{pmatrix}
            \beta_{1,1} & \beta_{1,2} & 0 & \ldots & \alpha_{1}\\
            \beta_{2,1} & \beta_{2,2} & d_{2,2} & \ldots & \alpha_{2}\\
            \cdot & \cdots & \cdots & \cdot\\
            \beta_{k,1} & d_{k,2} & \ldots & \beta_{k,k-1} & \alpha_{k}
        \end{pmatrix}\\
        = & \begin{pmatrix}
            d_{1,1}\beta_{1,1} & * & \ldots & *\\
            0 & d_{2,1} \beta_{1,2} + d_{2,2}\beta_{2,2} & \ldots & *\\
            \cdots & \cdots & \cdots & \cdot\\
            0  & 0 & \ldots & N_{k}
        \end{pmatrix}.
    \end{aligned}
    \label{eq: zonal11}
\end{equation}
The triangular matrix in the right-hand side is nonsingular, since by (\ref{eq: zonal9}) its diagonal elements are nonvanishing. Since the matrix $D$ is the transpose of the matrix $C$ and is therefore nonsingular, it follows that the matrix
$$
B_{k}=\left(\begin{array}{lllll}
\beta_{1,1} & \ldots & \beta_{1, k-1} & a_{1} \\
\cdot  & \cdots & \cdots & \cdot \\
\beta_{k, 1} & \ldots & \beta_{k, k-1} & \alpha_{k}
\end{array}\right)
$$
is also nonsingular. From (\ref{eq: zonal11}) we have
$$
\left(d_{k, 1}, d_{k, 2}, \ldots, d_{k, k}\right)=\left(0,0, \ldots, 0, N_{k}\right) B_{k}^{-1}
$$

This proves that for any $f$ the $d_{f, g}$ are expressed in terms of $\beta_{f, g}$ and $\alpha_{f}$. It remains to set forth an effective method of calculation of $\beta_{f, g}$ and $\alpha_{f}$. In general, it's hard to obtain the value of $\beta_{f, g}$ but we will restrict ourself in a narrow setting.

\subsection{Zonal polynomials}
Choose appropriately in (\ref{eq: congruence})
\[B = USU^{\prime}\]
with real orthogonal matrix $U$ which transforms real symmetric, positive definite matrix $S$ into diagonal matrices $B = (\beta_{i})_{1}^{n}$. Our goal is to find explicit expressions of $\Psi_{f_{1},f_{2},\dots,f_{n}}(S) = \Psi_{f_{1},f_{2},\dots,f_{n}}(U^{\prime}B^{\frac{1}{2}} U, U^{\prime}U^{\frac{1}{2}}U)$ for real symmetric, positive definite matrices $S$ such that 
\[\int_{\mathfrak{C}_{\text{II}}} \Psi_{f_{1},f_{2},\dots,f_{n}}(SHTH) \dot{H} = \frac{1}{V(\mathfrak{C}_{\text{II}})}\Psi_{f_{1},f_{2},\dots,f_{n}}(S)\Psi_{f_{1},f_{2},\dots,f_{n}}(T).\]
with another symmetric, positive definite matrix $T$.

These functions are invariant under the action of orthogonal matrices, i.e., 
$$\Psi_{f_{1},f_{2},\dots,f_{n}}(B) = \Psi_{f_{1},f_{2},\dots,f_{n}}(S)$$
due to properties of invariant subspaces (\ref{eq: plethysm}). By similar discussion before (\ref{eq: zonal3}), the summation of $\Psi_{f_{1},f_{2},\dots,f_{n}}(S)$ over all partitions should be required as a constant, so that they are uniquely determined. 

Since monomials $\beta^{[f]} = \beta_1^{f_1}\beta_{2}^{f_{2}} \dots \beta_n^{f_n}$ are linearly independent, homogeneous, symmetric bases in $\mathfrak{C}_{\text{II}}$ and any homogeneous polynomial of degree $f$ can be written in the form of a linear combination of these bases. (See discussions in section 4.1.) We aim to calculate the coefficients of $\Psi_{f_{1},f_{2},\dots,f_{n}}(S)$ under these bases $\beta^{[f]}$. We shall call them zonal polynomials.

\subsection{Relations to partial differential equations}

The Laplace-Betlami equation 
\[\sum_{\alpha,\beta = 1}^{n}\sum_{j,k = 1}^{n}\left(\delta_{\alpha\beta} - \sum_{l=1}^{n} z_{l\alpha} \bar{z}_{l\beta}\right)\left( \delta_{jk} - \sum_{\gamma=1}^{n} z_{j\gamma} \bar{z}_{k \gamma}\right) \frac{\partial^{2} u}{\partial z_{j\alpha} \partial \bar{z}_{k\beta}} = 0\]	
has a solution determined by the spherical harmonics up to a constant.

Since these polynomials are always combinations of monomials of latent roots. The part of the Laplace-Beltrami operator concerned with the diagonal elements (roots) becomes
\[
\Delta = \sum_{i=1}^{n} \left( y_i^2 \frac{\partial^2}{\partial y_i^2} - (n - 3) y_i \frac{\partial}{\partial y_i} \right) + \sum_{i \neq j} \frac{y_i^2}{y_i - y_j} \frac{\partial}{\partial y_i}.
\]
This operator governs the behavior of the eigenfunctions, including zonal polynomials. The eigenvalue for this equation is given by
\[
\lambda = \sum_{i=1}^{n} k_i (k_i + n - i -1 ),
\]
where \(k_i\) are the exponents in the monomial representing the highest weight term of the zonal polynomial.

\section{An explicit formula for zonal polynomials}

Consider the integral with respect to orthogonal mattrix $Q$
	\begin{equation}
	    \frac{1}{V(\mathfrak{C}_{\text{II}})}\int_{\mathfrak{C}_{\text{II}}} \left[\operatorname{tr} D_{\beta} Q D_{l} Q^{\prime}\right]^{f} \dot{[O]},
	\end{equation}
 where $D_{\beta} = \diag (\beta_{i})_{1}^{m}$ and $D_{l} = \diag (l_{i})_{1}^{m}$ and
	\[
	V(\mathfrak{C}_{\text{II}}) = \frac{\pi^{n(n+1)/4}}{\prod_{i=1}^{n} \Gamma\left(\frac{n+1-i}{2}\right)}.
	\]
	The trace \(\operatorname{tr} D_{\beta} Q D_{l} Q^{\prime}\) can be expanded as follows
	\[\begin{aligned}
		\operatorname{tr} D_{\beta} Q D_{l} Q^{\prime} &= \beta_{1} l_{1} \cos^{2} \theta_{11} + \beta_{2} l_{1} \sin^{2} \theta_{11} \cos^{2} \theta_{22} \\
		& + \beta_{3} l_{1} \sin^{2} \theta_{11} \sin^{2} \theta_{22} \cos^{2} \theta_{33} + \cdots \\
		& + \beta_{1} l_{2} \sin^{2} \theta_{11} \cos^{2} \theta_{12} + \beta_{2} l_{2} \left(\cos^{2} \theta_{11} \cos^{2} \theta_{12} \cos^{2} \theta_{22} + \right.\\
            &\left. \quad \sin^{2} \theta_{12} \sin^{2} \theta_{22} \cos^{2} \theta_{23}
		- 2 \cos \theta_{11} \cos \theta_{12} \sin \theta_{12} \cos \theta_{22} \sin \theta_{22} \cos \theta_{23}\right)  \\
        & + \cdots
	\end{aligned}\]
	The trace can be viewed as a polynomial in \(\cos(\theta_{ij})\) and \(\sin(\theta_{ij})\). Notably, terms with odd powers of \(\cos(\theta_{ij})\) vanish upon integration. The remaining terms, examined within specific bounds, yield a homogeneous and symmetric polynomial of degree \(f\) in both \(\beta\) and \(l\). Let $l^{[f]}$ be monomial polynomials $\sum l_{1}^{f_{1}}l_{2}^{f_{2}}\dots l_{n}^{f_{n}}$
    \begin{equation}
    a_{f;f} \beta^{[f]} l^{[f]} +a_{f; f-1,1}\left[\beta^{[f]}l^{[f-1,1]}+ \beta^{[f-1,1]}  l^{[f]} \right]+\cdots
    \end{equation}
    From the deduction of section 4.4, the zonal polynomials are expressed as
    \begin{equation}
        Z_{f}(l) = \sum_{h \leq f} b_{f,h}l^{[h]}
    \end{equation}
Upon integration, the coefficients $a$ are indeed
		\[
		\begin{aligned}
			a_{f; f} &= \frac{(2f-1)!!}{n(n+2) \cdots (n+2f-2)}, \\
			& \vdots \\
			a_{f; 1^{f}} &= \frac{f!}{n(n+2) \cdots(n+2f-2)}.
		\end{aligned}
		\]
Note that all denominators 
\begin{equation}
    c_{n}= \frac{(n+2f-2)!!}{(n-2)!!}
\end{equation}
are equal and no numerators include $n$. So (4.17) is now rewritten as 
\begin{equation}
    \frac{1}{(2 f-1)!!}\left[\frac{Z_{(f)}(\beta) Z_{(f)}(l)}{c_{n}} +\left\{a^{\prime}_{f-1,1; f-1,1} \beta^{[f-1,1]}l^{[f-1,1]}+\cdots\right\}\right]
\end{equation}
where $Z_{(f)}$ equals to
$$
Z_{(f)}(l)=(2 f-1)!! l^{[f]}+\cdots+f!l^{[1^f]}
$$
The expression in the bracket $\left\{\right\}$ includes neither $\beta^{(f)}$ nor $l^{(f)}$, and the coefficients $a^{\prime}$ are given by $g_{1} + g_{2} + \dots + g_{n} = h_{1} + h_{2} + \dots + h_{n} = f$,
    \[a^{\prime}_{g;h} =(2f-1)!! a_{g;h} - b_{f; g} b_{f; h}.\]
By the same procedure, we can determine, one after another, polynomials with their own multipliers. We summarise this finding as 
\begin{theorem} The coefficients $b$ in (5.3) are the same as those coefficients $a$ through evaluating the integral (5.2) by multiplying a common factor $c_{n}$. By assuming $b_{g;1^{f}} = f!$, these polynomials (5.3) are uniquely determined.
\end{theorem}
\begin{proof} Existence.
    Let $t$ be an integer such that $h_{1} = g_{1}\geq \dots\geq h_{i} = g_{i} + t\geq \dots\geq h_{j} = g_{j} - t\geq \dots\geq h_{n} = g_{n}$ is below or equal to $f_{1} \geq f_{2} \geq \dots \geq f_{n}$ for some $i,j$. We introduce $A_{f} = \sum_{i=1}^{n}f_{i}^{2}$ and $B_{f} = \sum_{i=1}^{n}if_{i}$,
\[b_{f,g} = \sum b_{f,h}\frac{g_{i}-g_{j} + 2t}{A_{f} - A_{g} + B_{g}- B_{f}}, \quad \]
where the summation is taken over all $i,j,t= 1,2,\dots, (l_{i-1} - l_{i})\land(l_{j} - l_{j+1})$. By dividing them by a factor $\frac{(n+2f-2)!!}{(n-2)!!}$, we find from the definition (5.2) this recursion holds also for $a$.

Uniqueness. The normalization is to guarantee all partitions can be shifted.
\end{proof}

\section{Applications}
	\subsection{Distribution of latent roots of the non-central Wishart distribution}
	
	Let $X$ be a symmetric matrix having the density element
	$$
	c|X|^{\frac{1}{2}(n-m-1)} \exp \left(-\frac{1}{2} \operatorname{tr} \Sigma^{-1} X\right) d X
	$$
	where
	\[c^{-1} = 2^{\frac{mn}{2}} \pi^{\frac{1}{4}m(m-1)}\prod_{i=1}^{m} \Gamma\left(\frac{n+1-i}{2}\right) |\Sigma|^{\frac{n}{2}}\]
	The substitution of
	$$
	X = U D_{l} U^{\prime}
	$$
	gives
	$$
	c \prod l_{i}^{\frac{1}{2}(n-m-1)} \prod(l_{i}-l_{j}) \prod d l_{i} \cdot \exp \left(-\frac{1}{2} \operatorname{tr} \Sigma^{-1} U D_{l} U^{\prime}\right) d U.
	$$
	Since $\operatorname{tr}\left(U D_{l} U^{\prime}\right)=\sum l_{i}$, for an arbitrary positive number $\lambda$,
	$$
	\begin{aligned}
		-\frac{1}{2} \operatorname{tr} \Sigma^{-1} U D_{l} U^{\prime} 
		& = -\frac{1}{2 \lambda} \sum l_{i}+\frac{1}{2} \operatorname{tr}\left(\frac{1}{\lambda} I-\Sigma^{-1}\right) U D_{l} U^{\prime}\\
		& =-\frac{1}{2 \lambda} \sum l_{i}+\frac{1}{2} \operatorname{tr} D_{\beta} Q D_{l} Q^{\prime}
	\end{aligned}$$
	where $\beta_{i}=\frac{1}{\lambda}-\frac{1}{\lambda_{i}^{-}}$ and $\lambda_{i}$ are the latent roots of $\Sigma$, and $Q=V^{\prime} U, V$ being a constant orthogonal matrix such that $\Sigma=V^{\prime} D_{\lambda} V$. Applying Theorem 3.2, the above is transformed into
	$$
	\begin{aligned}
		& c\left[\prod l_{i}^{(n-m-1)/2} \prod(l_{i}-l_{j}) \exp \left(-\frac{1}{2 \lambda} \sum l_{i}\right) \prod d l_{i}\right] \\
		& \cdot\left[\exp \left(\frac{1}{2} \operatorname{tr} D_{\beta} Q D_{l} Q^{\prime}\right) \prod \sin  (\theta_{i j})^{m-j-1} \prod d \theta_{i j}\right],
	\end{aligned}
	$$
	where $Q = Q(\theta_{i j})$. To obtain the marginal distribution, let us integrate with respect to $\theta$. Expanding $\exp \left(\frac{1}{2} \operatorname{tr} D_{\beta} Q D_{l} Q\right)$, we find the last component equals to
	$$
	\sum_{f=0}^{\infty} \frac{1}{f!2^{f}} \left[\operatorname{tr} D_{\beta} Q D_{i} Q\right]^{f} \prod \sin (\theta_{i j})^{m-j-1} \prod d \theta_{i j}
	$$
	Integrating term by term, we obtain 
	$$
	\frac{\pi^{\frac{1}{2}m(m+1)}}{\prod_{i=1}^{m} \Gamma\left[\frac{1}{2}(m+1-i)\right]} \sum_{f=0}^{\infty} \frac{1}{(2 f!)} \sum \frac{N(2f_{1},2f_{2},\dots, 2f_{m})}{Z_{f}(l)} Z_{f}(\beta) Z_{f}(l)
	$$
	where $Z_{f}(l)$ are zonal polynomials of $f$th degrees.
	
\appendix

\section{Zonal polynomials up to $f=6$ degrees}

\begin{table}[]
    \centering
    \caption{Zonal polynomials up to $f=6$ degrees.}
\begin{tabular}{|c|c|c|c|}
\hline $f$ & $\kappa$ & Zonal polynomial $Z_\kappa(A)$ & $\chi_{[2k]}(1)$  \\
\hline 1 & (1) & $s_1$ & 1 \\
\hline 2 & \begin{tabular}{l}
$(2)$ \\
$\left(1^2\right)$
\end{tabular} & \begin{tabular}{l}
$s_1^2+2 s_2$ \\
$s_1^2-s_2$
\end{tabular} & \begin{tabular}{l}
1 \\
2
\end{tabular} \\
\hline 3 & \begin{tabular}{c}
$(3)$ \\
$(21)$ \\
$\left(1^3\right)$
\end{tabular} & \begin{tabular}{l}
$s_1^3+6 s_1 s_2+8 s_3$ \\
$s_1^3+s_1 s_2-2 s_3$ \\
$s_1^3-3 s_1 s_2+2 s_3$
\end{tabular} & \begin{tabular}{l}
1 \\
9 \\
5
\end{tabular} \\
\hline 4 & \begin{tabular}{r}
$(4)$ \\
$(31)$ \\
$\left(2^2\right)$ \\
$\left(21^2\right)$ \\
$\left(1^4\right)$
\end{tabular} & \begin{tabular}{l}
$s_1^4+12 s_1^2 s_2+12 s_2^2+32 s_1 s_3+48 s_4$ \\
$s_1^4+5 s_1^2 s_2-2 s_2^2+4 s_1 s_3-8 s_4$ \\
$s_1^4+2 s_1^2 s_2+7 s_2^2-8 s_1 s_3-2 s_4$ \\
$s_1^4-s_1^2 s_2-2 s_2^2-2 s_1 s_3+4 s_4$ \\
$s_1^4-6 s_1^2 s_2+3 s_2^2+8 s_1 s_3-6 s_4$
\end{tabular} & \begin{tabular}{r}
1 \\
20 \\
14 \\
56 \\
14
\end{tabular}\\
\hline
\end{tabular}

\begin{tabular}{|c|c|c|}
\hline 
5 &  Zonal polynomial $Z_\kappa(A)$ &  $\chi_{[2k]}(1)$ \\
\hline \begin{tabular}{r}
$(5)$ \\
$(41)$ \\
$(32)$ \\
$\left(31^2\right)$ \\
$\left(2^2 1\right)$ \\
$\left(21^3\right)$ \\
$\left(1^5\right)$
\end{tabular} & \begin{tabular}{r}
$s_1^5+20 s_1^3 s_2+60 s_1 s_2^2+80 s_1^2 s_3+160 s_2 s_3+240 s_1 s_4+384 s_5$ \\
$s_1^5+11 s_1^3 s_2+6 s_1 s_2^2+26 s_1^2 s_3-20 s_2 s_3+24 s_1 s_4-48 s_5$ \\
$s_1^5+6 s_1^3 s_2+11 s_1 s_2^2-4 s_1^2 s_3+20 s_2 s_3-26 s_1 s_4-8 s_5$ \\
$s_1^5+3 s_1^3 s_2-10 s_1 s_2^2+2 s_1^2 s_3-4 s_2 s_3-8 s_1 s_4+16 s_5$ \\
$+5 s_1 s_2^2-10 s_1^2 s_3-10 s_2 s_3+10 s_1 s_4+4 s_5$ \\
$s_1^5-40 s_1^2 s_1^2 s_3^2+10 s_2 s_3+6 s_1 s_4-12 s_5$ \\
$s_1^5-4 s_1^3 s_2-3 s_1 s_2^2+20 s_1^2 s_3-20 s_2 s_3-30 s_1 s_4+24 s_5$ \\
$s_1^5-10 s_1^3 s_2+15 s_1 s_2^2+20 s_1^2-2$
\end{tabular} & \begin{tabular}{r}
1 \\
35 \\
90 \\
225 \\
252 \\
300 \\
42
\end{tabular}\\
\hline
\end{tabular}
\\
\begin{tabular}{|c|c|c|c|c|c|c|c|c|c|c|c|c|}
\hline6  & $s_1^6$ & $s_1^4 s_2$ & $s_1^2 s_2^2$ & $s_1^3 s_3$ & $s_2^3$ & $s_1 s_2 s_3$ & $s_1^2 s_4$ & $s_3^2$ & $s_2 s_4$ & $s_1 s_5$ & $s_6$ & $\chi_{[2 k]}(1)$ \\
\hline (6) & 1 & 30 & 180 & 160 & 120 & 960 & 720 & 640 & 1440 & 2304 & 3840 & 1 \\
\hline (51) & 1 & 19 & 48 & 72 & -12 & 80 & 192 & -64 & -144 & 192 & -384 & 54 \\
\hline (42) & 1 & 12 & 27 & 16 & 30 & 24 & -18 & -8 & 108 & -144 & -48 & 275 \\
\hline $(41^2)$ & 1 & 9 & -12 & 22 & -12 & -60 & 12 & 16 & -24 & -48 & 96 & 616 \\
\hline $(3^2)$ & 1 & 9 & 33 & -8 & -27 & 120 & -78 & 136 & -114 & -48 & -24 & 132 \\
\hline (321) & 1 & 4. & 3 & -8 & -2 & 0 & -18 & -24 & -4 & 32 & 16 & 2673 \\
\hline $(31^3)$ & 1 & 0 & -21 & 4 & 6 & 12 & -6 & 16 & 12 & 24 & -48 & 1925 \\
\hline $(2^3)$ & 1 & 0 & 15 & -20 & 30 & -60 & 30 & 40 & -60 & 24 & 0 & 462 \\
\hline $(2^2 1^2)$ & 1 & -3 & 3 & -8 & -9 & 0 & 24 & 4 & 24 & -24 & -12 & 2640 \\
\hline $(21^4)$ & 1 & -8 & 3 & 12 & 6 & 20 & -6 & -16 & -36 & -24 & 48 & 1485 \\
\hline $(1^6)$ & 1 & -15 & 45 & 40 & -15 & -120 & -90 & 40 & 90 & 144 & -120 & 132 \\
\hline
\end{tabular}
    \label{tab:my_label}
\end{table}

\bibliographystyle{plain}
\bibliography{inventiones}

\end{document}